\documentclass[12pt]{amsart}
\usepackage{amscd,amsmath,amsthm,amssymb,amstext}
\usepackage{mathtools}
\usepackage{mathrsfs}
\usepackage{comment}
\usepackage[all,tips]{xy}
\usepackage{color}
\usepackage{circuitikz}
\usepackage{caption}
\usepackage{enumitem} 

\definecolor{cadmiumgreen}{rgb}{0.0, 0.42, 0.24}
\usepackage[
colorlinks, citecolor=cadmiumgreen,
pagebackref,
pdfauthor={Farbod Shokrieh, Cameron Wright}, 
pdftitle={Torsor structures on spanning trees},
pdfstartview ={FitV},
]{hyperref}

\usepackage[
alphabetic,
backrefs,
msc-links,
nobysame,
lite,
]{amsrefs} 

\usepackage{tikz, float} 
\usetikzlibrary {positioning, arrows,arrows.meta}

\usetikzlibrary{calc,decorations.markings}
\usetikzlibrary{shapes,snakes}

\usepackage[left=3.5cm,top=3.5cm,right=3.5cm]{geometry}
\def\Div{{\rm Div}}
\def\Pic{{\rm Pic}}
\def\Prin{{\rm Prin}}

\def\A{{\mathcal A}}

\def\opn#1#2{\def#1{\operatorname{#2}}}
\opn\depth{depth} 
\opn\codim{codim}
\opn\ini{in} 
\opn\LM{LM}
\opn\LC{LC}
\opn\NF{NF}
\opn\Merge{Merge}
\opn\sgn{sgn}
\opn\div{div} 
\opn\Div{Div} 
\opn\Pic{Pic}
\opn\Prin{Prin}
\opn\Del{Del}
\opn\op{op}
\opn\ends{ends}
\opn\indeg{indeg} 
\opn\sign{sign}
\opn\outdeg{outdeg}
\opn\red{red}
\opn\Spec{Spec} 
\opn\Supp{Supp} 
\opn\supp{supp} 
\opn\Ker{Ker} 
\opn\Coker{Coker} 
\opn\Hom{Hom}
\opn\Tor{Tor} 
\opn\id{id}
\opn\span{span}
\opn\Image{Image}
\opn\con{conv} 
\opn\relint{rel.int} 
\opn\vol{vol}
\opn\val{val}
\opn\Zh{Zh}
\opn\Ber{Ber}
\opn\Vor{Vor}
\opn\Vol{Vol}
\opn\Covol{Covol}
\opn\Jac{Jac}
\opn\Dir{Dir}
\opn\can{can}
\opn\syz{{\rm syz}}
\opn\spoly{{\rm spoly}}
\opn\LM{{\rm LM}}
\opn\lm{{\rm lm}}
\opn\lcm{{\rm lcm}} 
\opn\A{\mathcal A}
\opn\dist{dist}
\opn\pd{pd}
\opn\en{en}
\opn\PL{PL}
\opn\dmeasz{DMeas_0}
\opn\dmeas{DMeas}
\opn\T{T}
\opn\circu{circ}
\opn\cocirc{cocirc}
\opn\Proj{Proj}

\def\Implies{\ifmmode\Longrightarrow \else
        \unskip${}\Longrightarrow{}$\ignorespaces\fi}
\def\implies{\ifmmode\Rightarrow \else
        \unskip${}\Rightarrow{}$\ignorespaces\fi}
\def\iff{\ifmmode\Longleftrightarrow \else
        \unskip${}\Longleftrightarrow{}$\ignorespaces\fi}

\newtheorem{Theorem}{Theorem}[section]
\newtheorem{Lemma}[Theorem]{Lemma}
\newtheorem{Corollary}[Theorem]{Corollary}
\newtheorem{Proposition}[Theorem]{Proposition}

\newtheorem{Conjecture}[Theorem]{Conjecture}
\newtheorem*{Conj}{Conjecture}

\theoremstyle{remark}

\theoremstyle{definition}
\newtheorem{Example}[Theorem]{Example}
\newtheorem{Definition}[Theorem]{Definition}

\def\qed{\ifhmode\textqed\fi
      \ifmmode\ifinner\quad\qedsymbol\else\dispqed\fi\fi}
\def\textqed{\unskip\nobreak\penalty50
       \hskip2em\hbox{}\nobreak\hfil\qedsymbol
       \parfillskip=0pt \finalhyphendemerits=0}
\def\dispqed{\rlap{\qquad\qedsymbol}}

\numberwithin{equation}{section}

\tikzstyle{Cwhite}=[scale = .8,circle, fill = white, minimum size=3mm] 
\tikzstyle{Cgray}=[scale = .4,circle, fill = gray, minimum size=3mm] 
\tikzstyle{Cblack2}=[scale = .4,circle, fill = black, minimum size=5mm] 
\tikzstyle{Cblack}=[scale = .7,circle, fill = black, minimum size=3mm]
\tikzstyle{C0}=[scale = .9,circle, fill = black!0, inner sep = 0pt, minimum size=3mm]
\tikzstyle{C1}=[scale = .7,circle, fill = black!0, inner sep = 0pt, minimum size=3mm]
\tikzstyle{Cred}=[scale = .4,circle, fill = red, minimum size=3mm]

\ctikzset{bipoles/resistor/height=0.15}
\ctikzset{bipoles/resistor/width=0.4}

\pagecolor{white}

\begin{document}

\title{Torsor structures on spanning trees}

\author{Farbod Shokrieh}
\email{\href{mailto:farbod@uw.edu}{farbod@uw.edu}}

\author{Cameron Wright}
\email{\href{mailto:wrightc8@uw.edu}{wrightc8@uw.edu}}



\date{\today}

\begin{abstract}
We study two actions of the (degree 0) Picard group on the set of the spanning trees of a finite ribbon graph. It is known that these two actions, denoted $\beta_q$ and $\rho_q$ respectively, are independent of the base vertex $q$  if and only if the ribbon graph is planar. Baker and Wang conjectured that in a nonplanar ribbon graph without multiple edges there always exists a vertex $q$ for which $\rho_q\neq\beta_q$. We prove the conjecture and extend it to a class of ribbon graphs with multiple edges. We also give explicit examples exploring the relationship between the two torsor structures in the nonplanar case.
\end{abstract}

\maketitle

\setcounter{tocdepth}{1}
\tableofcontents

\section{Introduction} \label{sec:Intro}
\renewcommand*{\theTheorem}{\Alph{Theorem}}

The number of spanning trees of a finite loopless graph $G$ is equal to the determinant of the reduced combinatorial Laplacian of the graph $G$; this is the celebrated Kirchoff matrix-tree theorem. In addition to counting spanning trees of the graph, this determinant also counts the number of elements of a certain group, known as the (degree 0) \textit{Picard group}, critical group, or sandpile group of the graph. This group, denoted $\Pic^0(G)$, is defined as the group of chip-firing equivalence classes on the collection of divisors on the graph under addition. The relationship between $\Pic^0(G)$ and the collection $\mathcal{T}(G)$ of spanning trees of the graph has been the subject of research in recent years. \\

In addition to equality in cardinality of the set $\mathcal{T}(G)$ and the Picard group $\Pic^0(G)$, it is known that the set $\mathcal{T}(G)$ admits two distinct torsor structures for the Picard group; the Picard group of a graph acts simply and transitively on the set $\mathcal{T}(G)$. Two such actions, namely the Bernardi action $\beta_q$ and the rotor-routing action $\rho_q$, have been the subject of much study in recent years \cite{ccg, bw} (see also \cite{HP, genus}). To define these torsors, we work with graphs endowed with a \textit{ribbon structure}, that is a cyclic ordering of the edges incident to each vertex. By giving a ribbon structure to a finite graph, we determine a surface in which the graph can be considered as embedded; for this reason, ribbon structures were originally termed `combinatorial embeddings' of graphs. 

After fixing a ribbon structure on a graph $G$, the Bernardi torsor $\beta_q$ and the rotor-routing torsor $\rho_q$ on $\mathcal{T}(G)$ rely on the choice of a base vertex vertex $q$ of the graph. It was proved by Chan, Church, and Grochow that the torsor $\rho_q$ is independent of the initial data $q$ if and only if the ribbon graph is planar \cite{ccg}. Separately, Baker and Wang proved that the Bernardi torsor $\beta_q$ is also independent of the initial data $q$ if and only if the ribbon graph is planar. Moreover, Baker and Wang proved in \cite{bw} that in the case of a planar ribbon graph, the torsor structures coincide; the Bernardi and rotor-routing processes produce the same simply transitive group action. Baker and Wang include the following conjecture:

\begin{Conj}[Baker and Wang \cite{bw}]
Given a nonplanar ribbon graph $G$ with no multiple edges and no loops, there exists a vertex $q$ for which $\rho_q\neq \beta_q$. 
\end{Conj}

We resolve this conjecture and extend the result to a class of graphs having multiple edges but no loops. In particular, it is proven here that any nonplanar graph with or without multiple edges admits a vertex $q$ for which $\rho_q\neq\beta_q$ when endowed with any ribbon structure. The criterion for the presence of such a vertex $q$ in a nonplanar ribbon graph with multiple edges is the presence of what we call a \textit{proper witness pair}. In Theorem \ref{thm:propwit}, we prove that the presence of a proper witness pair in a nonplanar ribbon graph is sufficient to ensure the presence of a vertex $q$ for which $\rho_q\neq \beta_q$. Further, we show in Proposition \ref{prop:nonplanar} that any nonplanar graph, when endowed with any ribbon structure, must contain a proper witness pair. In particular, any nonplanar graph endowed with any ribbon structure contains a vertex at which the torsor structures disagree. We prove Theorem \ref{thm:nopropwit}, the Baker-Wang conjecture for loopless nonplanar ribbon graphs with no multiple edges admitting no proper witness pair. Finally, we consider two examples and some associated computations; in particular we exhibit a graph having distinct vertices $p$ and $q$ such that $\rho_p=\beta_p$ while $\rho_q\neq\beta_q$.  \\

In an independent work, Changxin Ding in \cite{ding} also studies the Baker-Wang conjecture in the case where no multiple edges are allowed. Our approach to the conjecture appears to be different; in particular, we avoid the use of `$H$-decompositions' for planar graphs and we work with ribbon graphs having multiple edges.

\medskip

\renewcommand*{\theTheorem}{\arabic{section}.\arabic{Theorem}}

\section{Background and terminology} \label{sec:background}

By a \textit{graph} we mean a finite connected multigraph with no loops; for a graph $G$, we denote by $V(G)$ its vertex set and by $E(G)$ its edge multiset. Unless otherwise specified, multiple parallel edges are allowed. A \textit{ribbon graph} is a finite graph together with a cyclic ordering of the edges around each vertex, referred to as a \textit{ribbon structure}. Note that a ribbon structure on a graph induces a natural ribbon structure on any graph minor. Given two edges $e_0,e_1$ incident to a given vertex $v$, we will write $e_0\prec e_1$ if $e_1$ immediately succeeds $e_0$ in the order around $v$. If a vertex $v$ has two edges $e_0$ and $e_1$, the \textit{interval} $[e_0,e_1]$ between $e_0$ and $e_1$ is the collection of edges $\{e_0,f_1,\ldots,f_k,e_1\}$, where
	\[e_0\prec f_1 \prec \cdots \prec f_k\prec e_1.\]
The \textit{length} of an interval $[e_0,e_1]$ is its cardinality as a set.

A finite ribbon graph is equivalent to an embedding of the underlying graph into a closed oriented surface such that the ribbon structure agrees with the orientation of the surface (see, e.g. \cite[Theorem~3.7]{Tho}). In all figures in this paper, the ribbon structure is assumed to be given by the counterclockwise orientation of the underlying surface. 

A \textit{path} $P$ in a ribbon graph $G$ is the image of a mapping $P_k\to G$ of the length-$k$ path into $G$ that is injective on edges but not necessarily on vertices. A \textit{cycle} in $G$ is the image of a mapping $C_k\to G$ of the length-$k$ cycle into $G$ that is injective on both edges and vertices. Let $C$ be a cycle and endow $C$ with an orientation. A vertex $v$ of $C$ is incident to precisely two edges of $C$, $e_{\text{in}}$ and $e_{\text{out}}$, where the labeling of these edges agrees with the orientation of $C$. An edge $e$ incident to $v$ is said to be to the $\textit{left}$ of $C$ if $e\in[e_{\text{out}},e_{\text{in}}]$ and to the $\textit{right}$ of $C$ if $e\in[e_{\text{in}},e_{\text{out}}]$. \\

Given a group $\Gamma$, we consider a group action of $\Gamma$ on a set $X$ as a function
	\[\Gamma\times X \longrightarrow X\]
We say that a group action is \textit{regular} if this mapping is an isomorphism of sets. Stated differently, an action of $\Gamma$ on $X$ is regular if $\Gamma$ acts simply and transitively on $X$. If the set $X=\{1,\ldots,N\}$, we can consider a group action as a homomorphism from $\Gamma$ to the symmetric group on $N$ symbols:
	\[\Gamma\longrightarrow S_N\]
All multiplication of permutations will be from right to left, so that for example we have $(123)(34) = (1234)$.

\section{Cycles and witnesses} \label{graphs}

Here we are concerned in particular with nonplanar graphs. Following \cite{ccg}, we discuss planarity and nonplanarity of ribbon graphs in terms of the absence or presence of nonseparating cycles. 

\begin{figure}[h]
\begin{center}
\includegraphics[width=6cm, height=3cm]{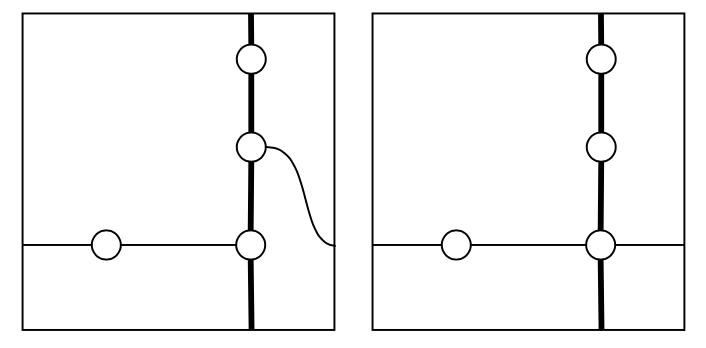} 
\end{center}
\caption{Nonseparating cycles (thickened) in two nonplanar ribbon graphs on the torus. At the left, we have a proper witness, while on the right we have an improper witness.}
\label{fig:nonsep}
\end{figure}

\begin{Definition} A cycle $C$ is \textit{nonseparating} if for any orientation of $C$ there exists a path $P$ that intersects $C$ only in the endpoints of $P$, such that the first edge of $P$ is on the left of $C$, and the last edge of $P$ is on the right of $C$. We will call $P$ a \textit{witness} for the nonseparating cycle $C$ and refer to the ordered pair $(C,P)$ as a \textit{witness pair}. 
\end{Definition}

\begin{Definition}
Given a nonseparating cycle $C$, we will say that a witness $P$ for $C$ is a \textit{proper witness} if the endpoints of $P$ are distinct vertices of $C$. In this case, we call the pair $(C,P)$ a \textit{proper witness pair} (see Figure \ref{fig:nonsep}).
\end{Definition}

A connected ribbon graph $G$ is \textit{nonplanar} if and only if it contains a nonseparating cycle. Note that this definition of nonplanarity is not equivalent to the underlying graph $G$ being nonplanar in the typical sense, but all nonplanar graphs will yield nonplanar ribbon graphs when endowed with any ribbon structure. A planar graph $G$ can be endowed with a ribbon structure such that the resulting ribbon graph is nonplanar. 

We have the following structural property of nonplanar graphs endowed with a ribbon structure.

\begin{Proposition}\label{prop:nonplanar}
A connected nonplanar graph $G$ endowed with any ribbon structure admits a \emph{proper} witness pair.
\end{Proposition}

\begin{proof}
Starting from an arbitrary ribbon graph $G$, first note that the operations of edge deletion and edge contraction cannot create a proper witness in $G$. Thus, by Kuratowski's theorem, it is enough to illustrate that $K_5$ and $K_{3,3}$ necessarily admit a proper witness pair when endowed with any ribbon structure.

Given any ribbon structure on $K_5$, there necessarily exists a nonseparating cycle; this cycle either includes all vertices or excludes at least one vertex. If the cycle contains all vertices, then any edge outside the cycle is a proper witness. Suppose then that the cycle excludes at least one vertex and has a witness containing a vertex $a$ on the cycle. Then some excluded vertex, say $b$, must be on a witness path; indeed, $b$ is adjacent to all vertices on the cycle and so any pair of edges connecting $b$ to vertices on the cycle yields a proper witness (see Figure \ref{fig:K5}). \\

\begin{figure}[h!]
\begin{center}
\includegraphics[width=8.5cm, height=2cm]{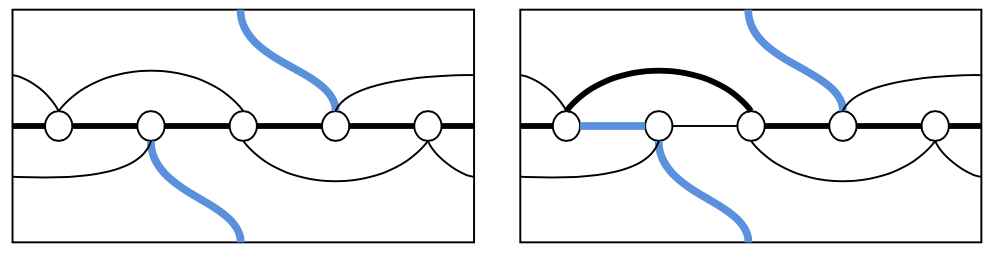} 
\end{center}
\caption{Two torus embeddings of $K_5$, each admitting proper witness pairs. In the left image, the thickened nonseparating cycle contains 5 vertices; in the right, the thickened cycle excludes one vertex. In each case, there is a proper witness path.}
\label{fig:K5}
\end{figure}

Consider now $K_{3,3}$. Suppose that the vertices of $K_{3,3}$ are labeled $\{a,b,c,x,y,z\}$, where the partition classes are $\{a,b,c\}$ and $\{x,y,z\}$. Since $K_{3,3}$ is nonplanar, any ribbon structure admits a nonseparating cycle $C$. If $C$ contains all vertices, then any witness is an edge outside the cycle and is therefore a proper witness. Suppose $C$ excludes at least one vertex. Since all cycles in bipartite graphs are of even length, any cycle in $K_{3,3}$ has at least two vertices from each partition class; suppose, without loss of generality, that $C$ contains the vertices $a$ and $b$. If some witness path consists of only one edge, this is a proper witness. In the case that no singleton witness path exists, some witness path $P$ must include some vertex $x$ in the other partition class. Now, since $x$ is adjacent to both $a$ and $b$, the path $P=\{\{a,x\},\{x,b\}\}$ is a proper witness for $C$, so that $(C,P)$ is a proper witness pair (see Figure \ref{fig:K33}).
\end{proof}

\begin{figure}[h]
\begin{center}
\includegraphics[width=10cm, height=2.5cm]{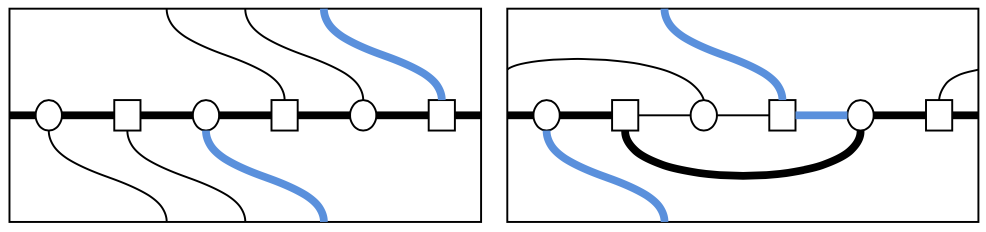} 
\end{center}
\caption{Two torus embeddings of $K_{3,3}$, each admitting proper witness pairs. In each, the two partition classes of vertices are represented by circles and squares. At the left, the thickened nonseparating cycle contains all 6 vertices; on the right, the thickened cycle contains only 4. Again, in each case the embedding admits a proper witness pair.}
\label{fig:K33}
\end{figure}

We next prove a similar result describing ribbon graphs admitting a nonplanar (as a ribbon graph) $K_4$ as a minor.

\begin{Proposition}
Suppose that $G$ is a graph containing the complete graph $K_4$ as a minor. If $G$ is endowed with a ribbon structure making the $K_4$ minor with the induced ribbon structure nonplanar, then $G$ admits a proper witness pair.
\end{Proposition}

\begin{figure}[h]
\begin{center}
\includegraphics[width=5cm, height=2.8cm]{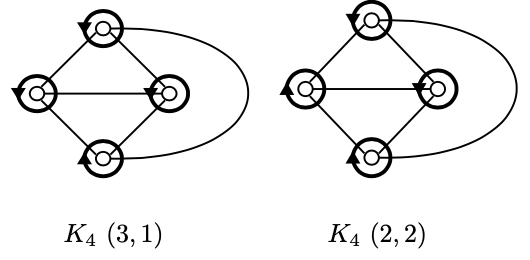} 
\end{center}
\caption{Two ribbon structures on $K_4$ yielding a nonplanar ribbon graph.}
\label{fig:K4ribbon}
\end{figure}

\begin{proof}
As in the proof of Proposition \ref{prop:nonplanar}, we begin with the observation that the presence of a nonseparating cycle with proper witness in the graph $K_4$ is sufficient to guarantee the presence of such a cycle with proper witness in any graph $G$ containing $K_4$ as a minor. It is therefore sufficient to show that each nonplanar ribbon structure on $K_4$ admits a proper witness pair. 

Since each vertex in $K_4$ is of degree 3, there are only two choices of cyclic ordering of the edges around each vertex (see Figure \ref{fig:K4ribbon}). In particular, we can fix a planar embedding of $K_4$ and enumerate all ribbon structures by choosing either a clockwise or counter-clockwise orientation of edges about each vertex. 

Moreover, since the automorphism group of $K_4$ is the full symmetric group on four elements, we need only consider the number of vertices oriented in either direction, rather than their relative positions. We refer to the ribbon structure in which $i$ vertices are oriented clockwise and $j$ are oriented counter-clockwise as being of type $(i,j)$. By reversing the orientation of the plane, it is clear that the presence of a proper witness pair in the ribbon structure of type $(i,j)$ ensures the presence of such a pair in the structure of type $(j,i)$. Now, since the ribbon structures of type $(4,0)$ and $(0,4)$ give rise to planar ribbon graph structures , we have only two choices of ribbon structures yielding a nonplanar ribbon graph $K_4$, arising respectively from the pairs in types $(3,1)$ and $(2,2)$. For each of these two ribbon structures, one can directly find the associated surface and verify that there exists a proper witness (see Figure \ref{fig:K4embedded}).
\end{proof}

\begin{figure}[h]
\begin{center}
\includegraphics[width=6.5cm, height=7cm]{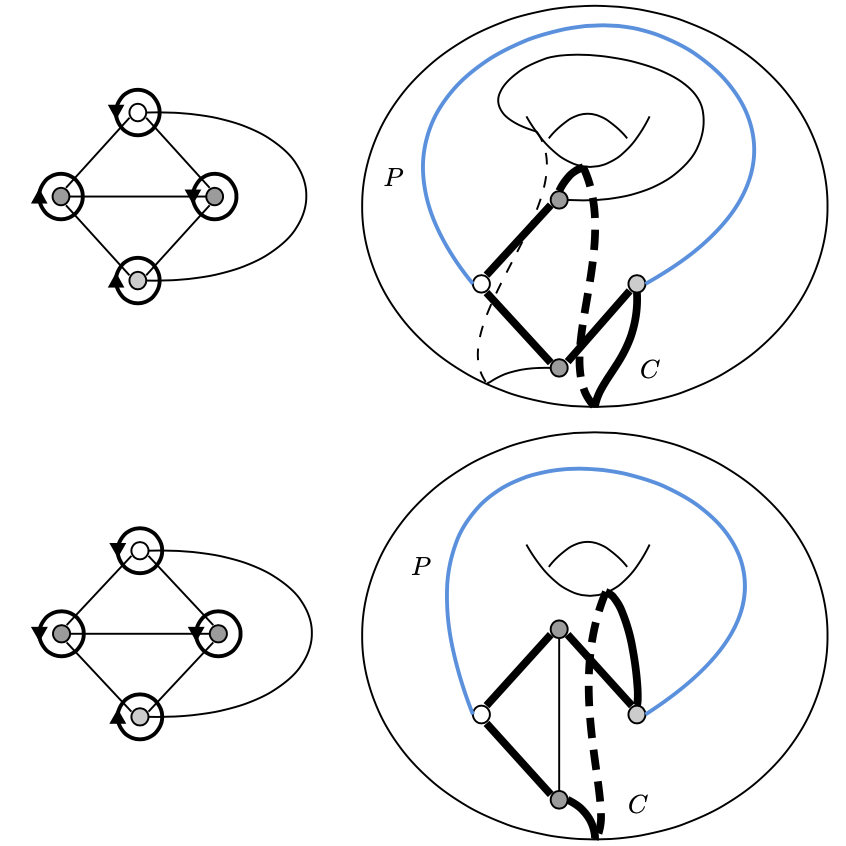} 
\end{center}
\caption{Proper witness pairs $(C,P)$ in each of the nonplanar ribbon structures on $K_4$. In each case, the cycle $C$ and its proper witness $P$ are given.}
\label{fig:K4embedded}
\end{figure}

In the case of a nonplanar ribbon graph which does not admit a proper witness pair, there is another kind of witness pair which will be of use to us.

\begin{Definition}\label{def:tight}
Let $G$ be a nonplanar ribbon graph that does not admit a proper witness pair. Suppose that $G$ admits a witness pair $(C,P)$ in $G$ such that $C$ and $P$ share a unique vertex $z$. We say that $(C,P)$ is a \textit{tight} witness pair if the edges $e$ and $f$ incident to $z$ in $C$ are such that the interval $[e,f]$ in the ribbon structure about $z$ has minimal length among all cycles $C$ passing through $z$ that have $P$ as a witness path.
\end{Definition}

Every finite nonplanar ribbon graph having no proper witness pair has at least one tight witness pair.

\section{Divisors on graphs}\label{divisors}

Given a graph $G$, let $\Div(G)$ denote the free abelian group with formal generators corresponding to the vertices of $G$. Elements of $\Div(G)$ are called \textit{divisors} on $G$; a given divisor $D$ has the form 
	\[D = \sum_{v\in V} a_v (v)\]
where $a_v\in \mathbb{Z}$ for all $v\in V$ and $(v)$ is the formal generator corresponding to $v$. It is often convenient to think of a divisor as an arrangement of \textit{chips} placed at the vertices of the graph $G$. The \textit{degree} of a divisor $D$ is the sum $\sum_{v\in V} a_v$. The collection of degree $d$ divisors is denoted by $\Div^d(G)$. Note that , given any vertex $q$ of $G$, the group $\text{Div}^0(G)$ is generated by the divisors $\{(v)-(q):\;v\in V\}$. 

Denote by $\mathcal{M}(G)$ the collection of functions $f:V\to\mathbb{Z}$. The \textit{combinatorial Laplacian} on $G$ is the mapping $\Delta:\mathcal{M}(G)\to\Div^0(G)$ given by \\
	\[\Delta f := \sum_{v\in V(G)} \bigg(\sum_{e=\{v,w\}} [f(v) - f(w)]\bigg) (v)\]
The image of $\Delta$ is the set of \textit{principal} divisors on $G$; such divisors form a subgroup $\Prin(G)$ of $\Div^0(G)$. We say that two divisors $D$ and $D'$ are \textit{linearly equivalent} if they differ by a principal divisor, so that $D'=D + \Delta f$ for some $f\in \mathcal{M}(G)$. The (degree 0) \textit{Picard group} of $G$ is the quotient group 
	\[\Pic^0(G) = \frac{\Div^0(G)}{\Prin(G)}\]
In general, we write $\Pic^d(G)=\Div^d(G)/\Prin(G)$; this is the collection of linear equivalence classes of degree-$d$ divisors. $\Pic^0(G)$ admits a regular action on $\Pic^d(G)$ for any $d$ given by addition of divisor classes. It is known that the cardinality of the Picard group of $G$ is equal to the number of spanning trees of the graph $G$ (see \cite{fm} and references therein). 

For a graph $G$, the \textit{genus} of $G$ is the first Betti number $g = |V(G)| - |E(G)| +1$; any spanning tree of $G$ has exactly $|E(G)| - g$ edges. Given a spanning tree $T\in\mathcal{T}(G)$ excluding edges $e_1,\ldots e_g$, a divisor $B\in \Div^g(G)$ is referred to as a \textit{break divisor} for $T$ if $B$ is of the form $B = \sum_{i=1}^g s(e_i)$, where $s(e_i)$ is one vertex of the edge $e_i$. Note that there are many possible distinct break divisors associated to a given tree.  We have the following property of break divisors: 

\begin{Proposition}[Theorem 1.1, \cite{abks}]\label{prop:breakdiv}
Each divisor class of $\Pic^g(G)$ contains a unique break divisor. 
\end{Proposition}

Consequently, we can use break divisors as canonical representatives of the equivalence classes in $\Pic^g(G)$.

\section{Bernardi and rotor-routing torsors}\label{torsors}

In addition to having the same cardinality as $\mathcal{T}(G)$, the Picard group $\Pic^0(G)$ also admits two regular group actions, or \textit{torsor structures}, on $\mathcal{T}(G)$. In particular we are concerned with two means of generating these actions; one is derived from a process due to Bernardi  and the other arises from a process referred to as rotor-routing (\cite{Bernardi}, \cite{HP}). We deal with torsor structures that depend on a choice of base vertex in the graph $G$. We remark that K\'{a}lm\'{a}n, Seunghun, and T\'{o}thm\'{o}r\'{e}sz have recently shown in \cite{tothmeresz} that, in the case of planar ribbon graphs, it is indeed possible to construct the Bernardi torsor and the rotor-routing torsor in a canonical way, without reference to a base vertex.

\subsection{Bernardi Torsor}

The first torsor structure considered in this work arise from a family of combinatorial maps, collectively referred to as the Bernardi process. As mentioned above, $\Pic^0(G)$ is in bijective correspondence with the set of spanning trees $\mathcal{T}(G)$. The works of Bernardi and Baker--Wang establish a family of bijections $\beta_{(q,e)}:\mathcal{T}(G)\to \Pic^g(G)$ witnessing this correspondence by mapping spanning trees to break divisors (\cite{Bernardi}, \cite{bw}). These maps are parametrized by pairs $(q,e)$ consisting of a vertex $q$ and an edge $e$ incident to $q$. 

\indent The Bernardi process uses the data $(q,e)$ and a spanning tree $T$ to perform a \textit{tour} $\tau_{(q,e)}(T)$ of the graph $G$. This tour can be thought of as a special walk in $G$ beginning and ending at $q$, traversing each edge of $T$ twice, and excluding each edge outside $T$. More specifically, the tour $\tau_{(q,e)}(T)$ takes the form $\tau_{(q,e)}(T) = (v_0,e_1,v_1,\ldots,e_k,v_k)$,
where $v_0=v_k=q$. 

Given $e_i,v_i$, the edge $e_{i+1}$ is chosen to be the first edge of $T$ succeeding $e_i$ in the ribbon structure about $v_i$; the process can be thought of as originating with the edge $e_0$ immediately preceding $e$ in the ribbon structure about $q$. We think of this process as starting with the initial data, traversing each edge of $T$ to the other endpoint, and cycling through the ribbon structure until finding another edge of $T$, ignoring each edge $f$ not included in $T$. In such a tour we ignore $g$ many edges $f_1,\ldots,f_g\notin T$, each being passed over two distinct times, once from each endpoint. 

\begin{figure}[h]
\begin{center}
\includegraphics[width=9cm, height=3.5cm]{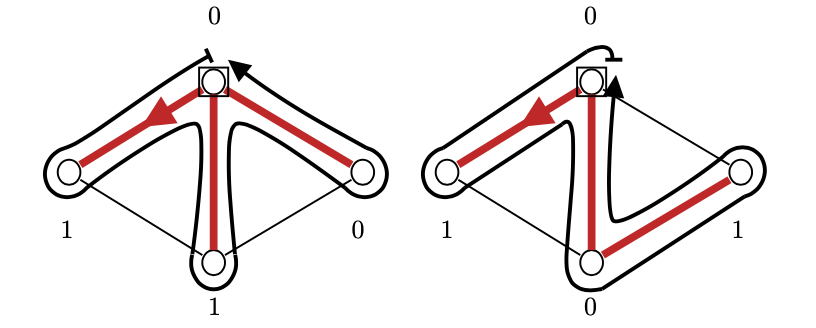} 
\end{center}
\caption{Bernardi tours and break divisors associated to two spanning trees of the given graph. Spanning trees are denoted by thickened red edges. In each case, the initial data $(q,e)$ is represented by a boxed vertex and an arrow along the initial edge.}
\label{fig:BernardiEx}
\end{figure}

In performing such a tour of $G$, we construct a divisor on $G$ by placing a chip at our feet the first time we ignore a given edge $f\notin T$. The divisor obtained in this fashion is necessarily a break divisor for the tree $T$. The mapping $\beta_{(q,e)}:\mathcal{T}(G)\to B(G)$ is given by mapping a spanning tree $T$ to the resulting break divisor. 

The process of deriving the break divisor $\beta_{(q,e)}(T)$ from $T$ is illustrated graphically in Figure \ref{fig:BernardiEx}. Since, by Proposition \ref{prop:breakdiv}, each divisor class of $\Pic^g(G)$ contains a unique break divisor, it follows that the maps $\beta_{(q,e)}$ induce explicit combinatorial bijections between spanning trees and divisor classes of degree $g$. 

From these bijections, Baker and Wang define a group action $\beta_q:\Pic^0(G)\times\mathcal{T}(G)\to\mathcal{T}(G)$ using the natural action of $\Pic^0(G)$ on $\Pic^g(G)$ \cite{bw}. Fix initial data $(q,e)$. Given any divisor class  $[D]\in\Pic^0(G)$, define 

\begin{equation}
\beta_q([D],T) :=\beta_{(q,e)}^{-1}([D]+[\beta_{(q,e)}(T)])  \label{eq:1}
\end{equation}

A central result of the paper of Baker and Wang states that the action \eqref{eq:1} is in fact independent of the chosen edge $e$ used in defining the bijection $\beta_{(q,e)}$ (see \cite[Theorem~4.1]{bw}). As a result of this, one can define a group action of $\Pic^0(G)$ on $\mathcal{T}(G)$ after fixing only a vertex of $G$, the \textit{base vertex} of the action. 

\subsection{Rotor-Routing Torsor}

Another regular action is derived from the \textit{rotor-routing} process on $G$, introduced in \cite{origRotor} and studied further in \cite{HP} and \cite{RR}. The rotor-routing process is a discrete-time dynamical system derived from $G$. The state space consists of \textit{rotor configurations} on $G$: pairs $(\sigma, v)$ where $v$ is a vertex of $G$, interpreted as the location of a chip, and $\sigma:V(G)\to E(G)$ is a \textit{rotor function} assigning to each vertex $w$ of $G$ an edge $\sigma(w)$ containing $w$. The edge $\sigma(w)$ is referred to as the \textit{rotor at w}. 

A single step of the rotor-routing process takes the rotor configuration $(\sigma, v)$ to $(\sigma',w)$, where $\sigma'$ is the rotor configuration for which $\sigma'(u)=\sigma(u)$ for all $u\neq v$, $\sigma'(v)$ is the successor of $\sigma(v)$ in the ribbon structure about $v$, and $w$ is the other vertex of $\sigma'(v)$. Informally, we move the rotor at $v$ once according to the ribbon structure about $v$ and move the chip along the new edge to its other vertex $w$. 

For the purposes of defining the action of $\Pic^0(G)$ on $\mathcal{T}(G)$, we choose a vertex $q$ of $G$ as a sink of the system and ignore the rotor at $q$. Given a spanning tree $T$ of $G$ and a vertex $q$ of $G$, we can obtain a rotor function $\sigma_T$ such that $\sigma_T(v)$ is the first edge on the unique path in $T$ from $v$ to $q$. Conversely, given a rotor function $\sigma$, we can obtain a subgraph $H_\sigma$ of $G$, whose only edges are the rotors of $\sigma$. 

\begin{figure}
\begin{center}
\includegraphics[width=9cm, height=3cm]{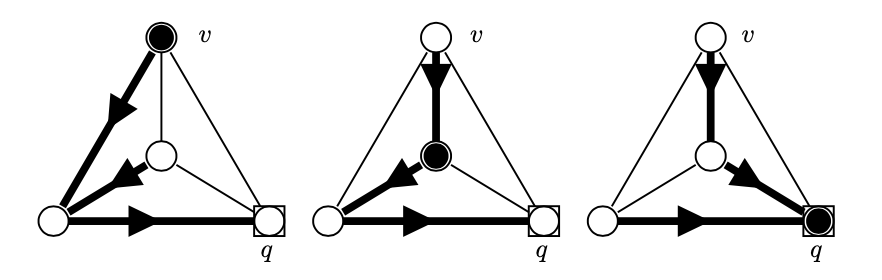} 
\end{center}
\caption{Rotor-routing process applied to a spanning tree of $K_4$. The original spanning tree $T$, indicated at the left by thickened edges, is mapped to the rightmost spanning tree $\rho_q([(v)-(q)])(T)$. At each stage, the location of the chip is indicated by a filled vertex.}
\label{fig:RRex}
\end{figure}

Given any vertex $q$ of $G$ and any spanning tree $T$ of $G$, iterating the rotor-routing process from any configuration $(\sigma_T,v)$ will eventually yield a configuration $(\tau,q)$; the chip will necessarily reach the vertex $q$ (\cite[Lemma~3.6]{RR}). Moreover, it can be shown that the graph $H_{\tau}$ derived from $\tau$ will be a spanning tree of $G$ (\cite[Lemma~3.10]{RR}). 

Using this observation, we can construct a group action $\rho_q$ of $\text{Div}^0(G)$ on $\mathcal{T}(G)$ by defining the action of generators $\{(v)-(q):v\in V(G)\}$ and extending linearly. Beginning from the rotor configuration $(\sigma_T,v)$, iterate the process until reaching some $(\tau,q)$ and let $T'$ be the spanning tree given by the final rotor function. This action is trivial on any divisor in $\text{Prin}(G)$ and so descends to an action by $\Pic^0(G)$. The rotor-routing action $\rho_q:\Pic^0(G)\times\mathcal{T}(G)\to\mathcal{T}(G)$ is then given by setting
	\begin{equation}
	\rho_q([(v)-(q)],T) = T' \label{eq:2}
	\end{equation}
See Figure \ref{fig:RRex} for an illustration of the rotor-routing action. In the investigation of \cite{ccg}, the following definition is of crucial use to establish the main findings. 

\begin{Definition}
A rotor configuration $(\sigma,v)$ is called a \textit{unicycle} if the image of the rotor function $\sigma$ contains a unique directed cycle and $v$ is a vertex of this cycle.
\end{Definition}

This definition will be of particular use to us due to the interaction between the rotor routing action and unicycles.

\begin{Lemma}[Lemma 4.9, \cite{HP}] \label{lem:unicycles}
Let $(\sigma,v)$ be a unicycle on a graph with $m$ edges. Then in iterating the rotor-routing process $2m$ times from $(\sigma,v)$, the chip traverses each edge of $G$ exactly once in each direction, each rotor makes exactly one full rotation, and the final state is again $(\sigma,v)$. 
\end{Lemma}

The following lemma is an easy consequence of Lemma \ref{lem:unicycles}:

\begin{Lemma}\label{lem:rotations}
Suppose that $G$ is a ribbon graph, $(\sigma,w)$ a unicycle on $G$, and $z$ a vertex satisfying $\sigma(z)=w$. Then, for all neighbors $v\neq w$ of $z$ satisfying $\sigma(v)=\{v,z\}$, the rotor at $z$ reaches $\{z,v\}$ before the rotor at $v$ completes a full rotation.
\end{Lemma}
\begin{proof}
Let $v$ be a neighbor of $z$ such that $\sigma(v)=\{v,z\}$. Suppose, for the sake of contradiction, that in iterating the rotor-routing process from $(\sigma,z)$, the rotor at $v$ completes a full rotation before the rotor at $z$ reaches $\{z,v\}$. Then in order for the rotor at $z$ complete a full rotation, we pass through some rotor configuration $(\rho,v)$ such that $\rho(z)=v$. Then iterating the rotor-routing process one additional time the rotor at $v$ has completed more than one full turn before the rotor at $z$ returns to its starting position, contradicting Lemma \ref{lem:unicycles}.
\end{proof}

\subsection{Comparison of Torsor Structures}\label{comparison}

The rotor-routing and Bernardi processes described in the previous subsections give rise to regular group actions of $\Pic^0(G)$ on $\mathcal{T}(G)$, yielding two torsor structures on $\mathcal{T}(G)$ via \eqref{eq:1} and \eqref{eq:2}. Each of the two structures is dependent on a choice of some vertex of $G$ as base vertex for the action. These torsor structures have been the subject of much study in recent years (\cite{bw}, \cite{ccg}, \cite{genus}, \cite{tothmeresz}).

Fix some vertex $q\in V(G)$ and consider $\rho_q$ and $\beta_q$. To see that the two homomorphisms $\rho_q$ and $\beta_q$ agree, we must have that
	\[\rho_q([D],T) = \beta_q([D],T)\]
for each $T\in\mathcal{T}(G)$ for each divisor class $[D]\in\Pic^0(G)$. In particular, since divisors of the form $(v)-(q)$ generate $\text{Div}^0(G)$, by \eqref{eq:1} and \eqref{eq:2}, it is sufficient to verify that for each $v\in V(G)$ we have the equality
	\[\beta_{(q,e)}(\rho_q([(v)-(q)],T)) = [(v)-(q)] + \beta_{(q,e)}(T)\]
In contrast, to show that $\rho_q\neq\beta_q$ we need only find a single tree $T\in\mathcal{T}$ for which the above equality does not hold; doing so immediately yields that as permutations of $\mathcal{T}(G)$,
\[\rho_q([(v)-(q)],\;\cdot\;) \neq \beta_q([(v)-(q)],\;\cdot\;)\] 
hence $\rho_q$ and $\beta_q$ cannot be equal. 

We have the following theorem about the rotor-routing torsor structure on $\mathcal{T}(G)$: 

\begin{Theorem}[Theorem 2, \cite{ccg}]
Let $G$ be a connected ribbon graph. The action $\rho_q$ of $\Pic^0(G)$ on $\mathcal{T}(G)$ is independent of the base vertex $q$ if and only if $G$ is a planar ribbon graph.
\end{Theorem}

The work of Baker and Wang established an analogous result for the torsor structure derived from the Bernardi process: 

\begin{Theorem}[Theorems 5.1 and 5.4, \cite{bw}]
Let $G$ be a connected ribbon graph. The action $\beta_q$ of $\Pic^0(G)$ on $\mathcal{T}(G)$ is independent of the base vertex $q$ if and only if $G$ is a planar ribbon graph.
\end{Theorem}

In the case of a planar ribbon graph $G$, neither the rotor-routing action nor the Bernardi action is dependent on the choice of base vertex; both actions are in this sense canonical. 

\begin{Theorem}[Theorem 7.1, \cite{bw}]
For a planar ribbon graph $G$, the Bernardi and rotor-routing processes define the same $\Pic^0(G)$-torsor structure on $\mathcal{T}(G)$.
\end{Theorem}
However, in the case of a non-planar $G$ the situation was not as clear. The work \cite{bw} concludes with the following conjecture: 

\begin{Conjecture}[Conjecture 7.2, \cite{bw}]
Let $G$ be a nonplanar ribbon graph with no multiple edges and no loops. Then there exists a vertex $q$ of $G$ such that $\rho_q\neq\beta_q$, 
\end{Conjecture}

\section{Baker-Wang Conjecture and proper witnesses}\label{main}

We now prove the Baker-Wang conjecture in the case that the graph $G$ admits a proper witness pair. This will be done by first proving the conjecture in a more restricted situation. 

\begin{Lemma}\label{lem:propwit}
Suppose that $G$ is a ribbon graph with proper witness pair $(C,P)$ such that $P$ has endpoints $x$ and $z$. If the two edges $e_0$ and $e_1$ of $C$ incident to $z$ satisfy $e_0\prec e_1$, then there exists a vertex $q$ of $G$ such that $\rho_q\neq\beta_q$.
\end{Lemma}
\begin{proof}
Suppose that $e_1=\{z,q\}$, so that $q$ is the other vertex of $e_1$ in $C$. Now, let $e'=\{u,v\}$ be an edge of $P$ and extend the collection of edges $(C\setminus e_1)\cup(P\setminus e')$ to a spanning tree $T$ of $G$. We will show that 
	\[\rho_q([(z)-(q)],T) \neq\beta_q([(z)-(q)],T)\] 
In particular, denoting by $T'$ the spanning tree $\rho_q([(z)-(q)],T)$, we show that the divisor classes $[(z)-(q)] + [\beta_{q,e_1}(T)]$ and $[\beta_{q,e_1}(T')]$ have distinct break divisor representatives. By the definition of the Bernardi action $\beta_q$, we immediately obtain $\rho_q\neq\beta_q$. 

\begin{figure}
\begin{center}
\includegraphics[width=8cm, height=5cm]{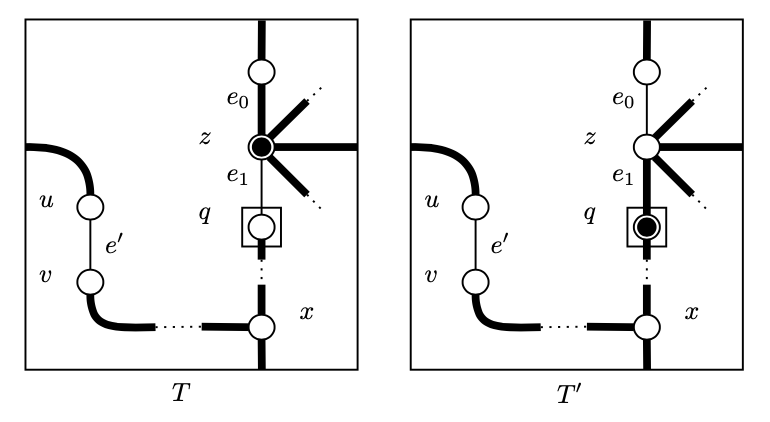} 
\end{center}
\caption{Rotor-routing process applied to the spanning tree $T$, thickened at the left, as in the proof of Lemma \ref{lem:propwit}. After one step the chip is moved from $z$ to $q$ and the process is terminated, yielding the thickened tree $T'$ at the right.}
\label{fig:TTprime}
\end{figure}

Acting on $T$ with $\rho_q([(z)-(q)])$ entails only a single step of the rotor routing process: changing the rotor at $z$ from $e_0$ to $e_1$ and moving the chip to $q$, terminating the process. As a result, the spanning trees $T$ and $T'$ differ only by a single edge. In particular, $T' = (T\setminus e_0)\cup e_1$. This situation is depicted in Figure \ref{fig:TTprime}. 

Consider the divisors $(z)-(q) + \beta_{q,e_1}(T)$ and $\beta_{q,e_1}(T')$. By the definition of the Bernardi bijection, we know that $\beta_{q,e_1}(T')$ is a break divisor for $T'$. On the other hand, the divisor $(z)-(q)+\beta_{q,e_1}(T)$ is also a break divisor for $T'$, distinct from $\beta_{q,e_1}(T')$. Since a break divisor for a spanning tree is a formal sum of endpoints of edges outside that tree, we can see the following:
	\[(z)-(q) + \beta_{q,e_1}(T) = (z)-(q) + (q) + \sum_{\substack{e\notin T \\ e\neq e_0}} s(e) = (z) + \sum_{\substack{e\notin T' \\ e\neq e_1}} s(e)\]
Since $z$ is an endpoint of $e_1$, this shows that $(z)-(q)+\beta_{q,e_1}(T)$ is a break divisor for $T'$. 

We are reduced to showing that $(z)-(q)+\beta_{q,e_1}(T)$ and $\beta_{q,e_1}(T')$ are distinct break divisors. By the construction of $T$, the collection of edges $P\cap T$ consists of two connected components $P_z$ and $P_x$, containing $z$ and $x$, respectively. In the Bernardi tour of $T$, $P_x$ is traversed before $P_z$, while the opposite is true for the tour of $T'$. This results in different placement of the chip associated to the edge $e'$ in the two tours: without loss of generality, the $T$ tour places a chip at $v$ while the $T'$ tour places a chip at $u$; this difference is depicted in Figure \ref{fig:difference}. Since $(z)-(q) + \beta_{q,e_1}(T)$ now differs from $\beta_{q,e_1}(T')$ in at least one place and both are break divisors, they cannot be representative of the same element of $\Pic^0(G)$. 
\end{proof}

\begin{figure}
\begin{center}
\includegraphics[width=8cm, height=5cm]{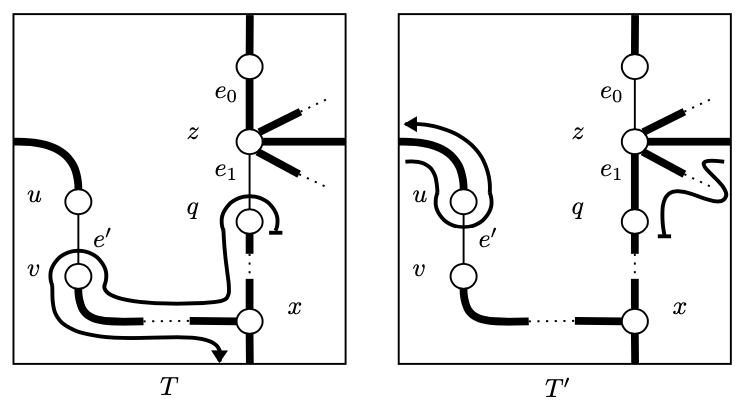} 
\end{center}
\caption{Bernardi process with initial data $(q,e_1)$ applied to the spanning trees $T$ and $T'$ as in the proof. For simplicity, only the relevant parts of the tours are included in the figure. At the left, the tour cuts $e'$ at $u$, while at the right $e'$ is cut at $v$.}
\label{fig:difference}
\end{figure}

\vspace{0.2cm}

For a ribbon graph $G$ admitting a proper witness pair, we either have $z$, $e_0$, and $e_1$ as in Lemma \ref{lem:propwit} above, or there are some edges $f_1,\ldots,f_k$ incident to $z$ such that 
\[e_0\prec f_1 \prec \cdots\prec f_k\prec e_1\]
We now prove that in the latter case, we can still find vertices $q'$ and $z'$ and an edge $e$ for which the  break divisors $\beta_{q',e}(T')$ and $(z')-(q') + \beta_{q',e}(T)$ still differ as described.

\begin{Theorem}\label{thm:propwit}
For any nonplanar ribbon graph $G$ admitting a proper witness pair $(C,P)$, there exists a vertex $q$ for which $\rho_q\neq \beta_q$. 
\end{Theorem}

\begin{proof} 
Suppose that the proper witness pair $(C,P)$ such that the path $P$ intersects $C$ in vertices $x$ and $z$. If this witness pair satisfies $e_0\prec e_1$ in the ribbon structure about $z$, then we are finished. Suppose also that the pair $(C,P)$ fails to meet the conditions of Lemma \ref{lem:propwit}. In particular, there exists an edge $f=\{z,v\}$ such that $f\in [e_0,e_1]$ and $f\neq e_0,e_1$. Let $G'$ be the subgraph of $G$ obtained by deleting the vertex $z$ and all edges to which it is incident. Denote by $F$ the connected component of the subgraph $G'$ containing $v$. See Figure \ref{fig:extraedge}. Note in particular that $F$ may intersect (and therefore contain) the cycle $C$ or the path $P$. 

\begin{figure}
\begin{center}
\includegraphics[width=7cm, height=5cm]{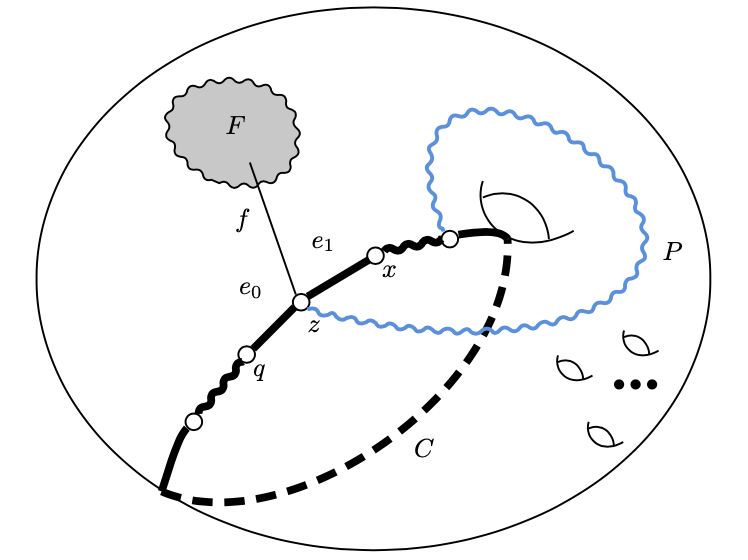} 
\end{center}
\caption{Proper witness pair $(C,P)$ in a graph embedded on a surface of positive genus. Supposing that there is an edge $f$ incident to $z$ with $e_0\prec f \prec e_1$, we can identify the component $F$ of $G\setminus z$ and consider its possible intersection with the cycle $C$ and the path $P$. Wavy lines represent paths that may contain additional vertices and edges.}
\label{fig:extraedge}
\end{figure}

If the component $F$ meets neither $C$ nor $P$, then one can ignore the edge $f$, applying the same reasoning as in Lemma \ref{lem:propwit}: the rotor-routing process carried out in the subgraph $F$ will not influence the rotors on the path $P$, so the break divisors $(z)-(q) + \beta_{q,e_1}(T)$ and $\beta_{q,e_1}(T')$ will still differ as described in the proof. We therefore assume that there is no such edge $f$ with $F\cap C$ and $F\cap P$ empty.

For each edge $f_i=\{z,v_i\}\in [e_0,e_1]$ such that $f_i\neq e_0,e_1$, suppose that $F_i$ is the connected component of $G'$ containing $v_i$. We prove the result by induction on the number $t$ of such edges $f_i$. By Lemma \ref{lem:propwit}, we have already established the base case $t=0$. Suppose then that for all $\ell\leq t-1$ we can find some vertex $q$ such that $\rho_q\neq\beta_q$.

Let $f=\{z,v\}$ denote the first edge after $e_0$ in $[e_0,e_1]$, and let $F$ denote the connected component of $G'$ containing $v$. In the event that $F$ intersects either the cycle $C$ or the path $P$, we will incorporate the edge $f$ into either a nonseparating cycle or a proper witness of a nonseparating cycle in such a way that we can apply the Lemma \ref{lem:propwit}. If $F$ intersects both, then either approach will suffice. 

\begin{itemize}[leftmargin=*]
\item Suppose $c$ is a vertex in the intersection of $F$ and $C$, so that we can extend $f=\{z,v\}$ to a path $Q$ from $z$ to $c$. Note that by construction there exists a path $C_{c,z}$ from $c$ to $z$ within the cycle $C$. Define a cycle $C':=(C\setminus C_{c,z})\cup Q$; it is necessarily nonseparating with the same proper witness path $P$. Now $f$ and $e_0$ are the edges of $C'$ incident to $z$ and $e_0\prec f$ around $z$, so that we can now apply Lemma \ref{lem:propwit} with proper witness pair $(C',P)$, allowing the vertex $v$ to fill the role of $q$. See Figure \ref{fig:FC}. 

\item Suppose that $x'$ is a vertex in the intersection of $F$ and $P$, then we can extend $f=\{z,v\}$ to a path $Q$ from $z$ to $x'$. There already exists a path in $P$ from $x'$ to $x$; let $P_{x,x'}$ denote this path and let $C_{z,x}$ denote the path in $C$ from $z$ to $x$. Then the cycle $C':=(C\setminus C_{z,x})\cup Q\cup P_{x,x'}$ is nonseparating with witness path $P':=P\setminus P_{x,x'}$. Moreover, $f$ and $e_1$ are the edges of $C'$ incident to $z$ and $f\prec e_1$ around $z$, so that we can now apply Lemma \ref{lem:propwit} with proper witness pair $(C',P')$. See Figure \ref{fig:FP}. \qed
\end{itemize}

\end{proof}

\begin{figure}
\begin{center}
\includegraphics[width=11cm, height=5cm]{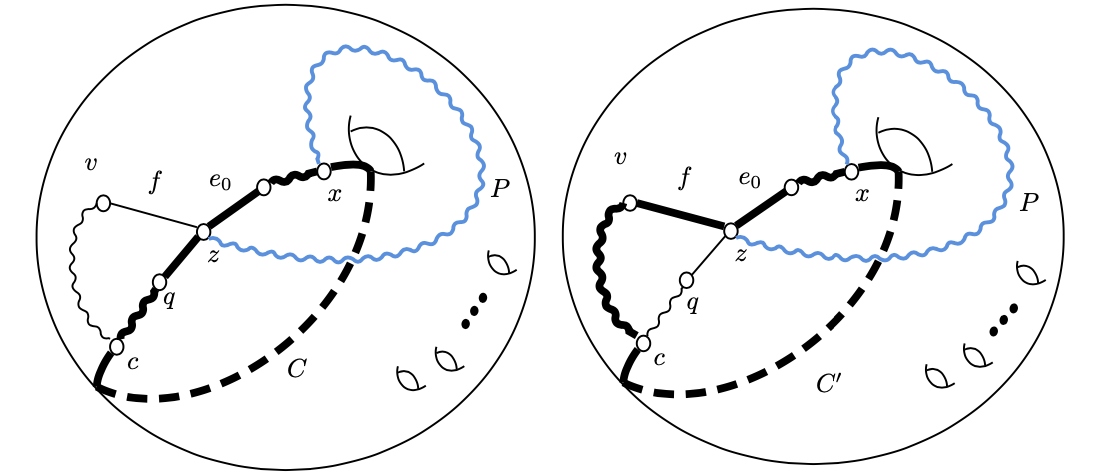} 
\end{center}
\caption{Situation in the proof of Theorem \ref{thm:propwit} where the intruding edge $f$ leads to a path to $C$. The new proper witness pair $(C',P)$ is shown in the image at the right.}
\label{fig:FC}
\end{figure}

\begin{figure}
\begin{center}
\includegraphics[width=11cm, height=5cm]{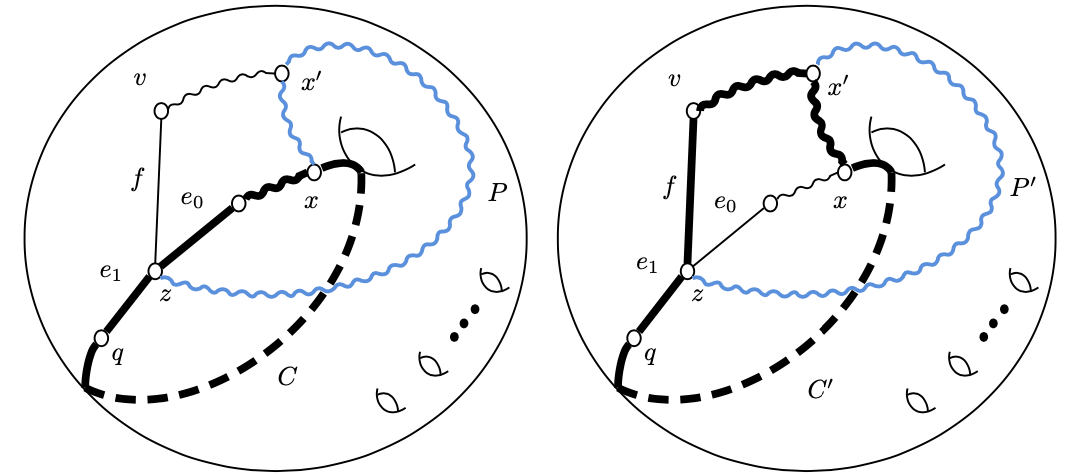} 
\end{center}
\caption{Situation in the proof of Theorem \ref{thm:propwit} where the intruding edge $f$ leads to a path to $P$. The new proper witness pair $(C',P)$ is shown in the right figure.}
\label{fig:FP}
\end{figure}

Considering Proposition \ref{prop:nonplanar} alongside the theorem above, we immediately obtain the following corollary:

\begin{Corollary}
For a nonplanar graph $G$ endowed with any ribbon structure, there exists some vertex $q$ of $G$ such that $\rho_q\neq\beta_q$.
\end{Corollary}

\section{Simple graphs without proper witnesses}\label{nopropwit}

We next show that for graphs $G$ without multiple edges or loops that do not contain a proper witness pair, there is some vertex $q$ for which $\rho_q\neq \beta_q$. This completes the proof of the Baker-Wang conjecture. 

\begin{Theorem}\label{thm:nopropwit}
Let $G$ denote a nonplanar ribbon graph with no multiple edges and no loops such that $G$ does not admit a proper witness pair. Then there is some vertex $q$ of $G$ such that $\rho_q\neq\beta_q$.
\end{Theorem}

\begin{proof}
As in the proof of Lemma \ref{lem:propwit}, we will construct a spanning tree $T$ such that the Bernardi and rotor-routing torsors disagree on $T$. 

In order to construct the spanning tree $T$ and to deduce disagreement of the two torsors at the desired vertex $q$, we must first fix some terminology regarding our construction. It is recommended that the reader reference Figure \ref{fig:LRH} while we introduce the required terminology.
\begin{enumerate}

\item Since $G$ is a nonplanar ribbon graph with no proper witness pair, there exists at least one tight witness pair (Definition \ref{def:tight}). Suppose $(C,P)$ is one such pair and that $z$ is the unique vertex in the intersection of $P$ and $C$.

\item Suppose that $q$ is a neighbor of $z$ in $C$. Orient $C$ from $q$ to $z$. We may assume that the interval on the left of $C$ is the minimal interval described in the definition of a tight witness pair. Let $e$ be the edge $\{z,q\}$ and let $e'$ be the edge immediately succeeding $e$ in the ribbon structure about $q$. 

\item Partition all neighbors $y$ of $z$ outside the cycle $C$ into two classes $L$ and $R$ so that $y\in L$ if and only if $\{y,z\}$ is on the left of $C$ and $y\in R$ if and only if $\{y,z\}$ is on the right of $C$.  

\item Let $Z\subseteq E(G)$ denote the collection of edges incident to $z$ with endpoints in $L$ and $R$. We can extend the collection of edges $C\cup Z\setminus \{e\}$ to a spanning tree $T$ so that all edges incident to $z$ except $e$ are included in the tree $T$. 

Let $H$ be the subgraph of $G$ induced by all witness paths for $C$ that intersect $C$ at $z$. Further, let $L_H\subseteq L$ and $R_H\subseteq R$ denote the vertices of $H$ which lie in $L$ and $R$, respectively. 

\item Let $w_0$ denote the vertex in $L_H$ such that $\{z,w_0\}$ is the first edge in $H$ to the left of $C$. Likewise, let $v_0$ denote the vertex in $R_H$ such that $\{z,v_0\}$ is the first edge to the right of $C$. 
 
 \item Consider the vertices of $w\in L\setminus L_H$. As in the proof of Theorem \ref{thm:propwit}, let $G'$ denote the subgraph of $G$ obtained by removing the vertex $z$ and all edges incident to $z$. Let $W$ denote the connected component of $G'$ containing $w$. Since $(C,P)$ is a tight witness pair,  $W$ does not intersect the cycle $C$.  
 \end{enumerate}

\begin{figure}
\begin{center}
\includegraphics[width=8cm, height=6cm]{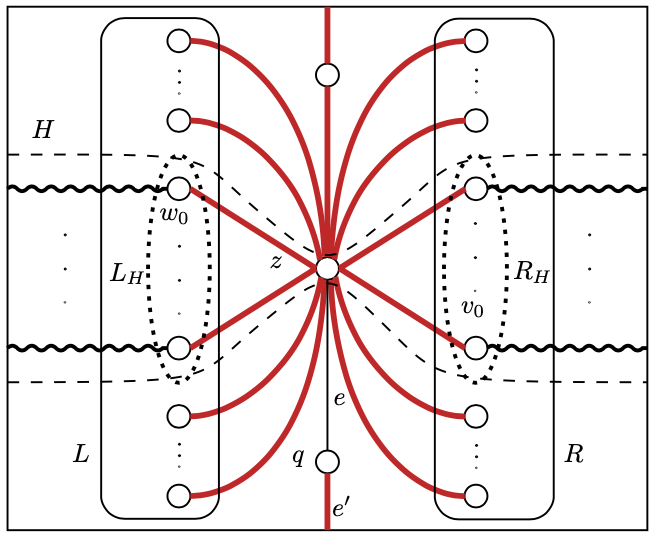}
\end{center}
\caption{The situation in the proof of Theorem \ref{thm:nopropwit}. The tree $T$ is indicated with thickened, while the subgraph $H$ of $G$ is indicated using a dashed line. The vertex sets $L_H$ and $R_H$ are shown within the dotted regions within $H$. For simplicity, we have depicted that the edge $e'$ as an edge of the cycle $C$.}
\label{fig:LRH}
\end{figure}

Consider now the rotor-routing process associated to $\rho_q([(z)-(q)],T)$. At some point, the rotor at $z$ shifts to the edge $\{z,w_0\}$, forming a unique cycle in the subgraph of $H$ defined by the rotors at vertices of $H$. Restricting the rotor configuration to the graph $H$, we obtain a unicycle $U=(\sigma,z)$ in the graph $H$. Further, by the assumption that $G$ has no proper witness, we see that all paths from any vertex of $H$ other than $z$ to any vertex on the path $C$ must pass through $z$. In particular, combining this with the fact that $(C,P)$ is a tight witness pair ensures that no rotors outside of $H$ or the subgraphs $W$ will move until the rotor at $z$ shifts to $e$. 
 
Consider now the rotor routing process in the graph $H$ beginning from starting configuration $U$. Then by Lemma \ref{lem:rotations}, for all neighbors $v$ of $z$ in $R_H$, the rotor at $z$ shifts to $\{z,v\}$ before the rotor at any $v\in R_H$ completes a full rotation. In particular, we necessarily obtain that the rotor at $z$ reaches $\{z,v_0\}$ before the rotor at any $v\in R_H$ completes a full rotation. Considering the implication of this statement in the graph $G$, we see that the rotor at $z$ must shift to $e$ before the rotors at any $v\in R_H$ complete a full rotation. 

We now see that the rotor routing process associated to $\rho_q([(z)-(q)],T)$ terminates before the rotor of any vertex in $R_H$ completes a full rotation. In particular, this means that in the tree $T':=\rho_q([(z)-(q)],T)$ does not contain any of the edges connecting $z$ to $R_H$. 

Now consider the Bernardi process associated to $T'$. By the above deliberation, in the Bernardi tour associated to $\beta_{(q,e)}(T')$, the vertex $z$ receives a chip from at least one neighbor in $R_H$ in addition to the chip it receives from the edge $\{z,x\}$. By the assumption that $G$ has no proper witness, this means that the break divisor $\beta_{(q,e)}(T')$ takes a value of at least two at $z$. 

On the other hand, by the construction of the tree $T$, the break divisor $\beta_{(q,e)}(T)$ takes the value 0 at $z$. This is because all edges incident to $z$, except $e$, are included in the tree $T$. It immediately follows that the divisor $(z)-(q)+\beta_{(q,e)}(T)$ takes the value 1 at $z$.  

Further we note that this divisor is a break divisor; in fact $(z)-(q)+\beta_{(q,e)}(T)=\beta_{(q,e')}(T)$. We have deduced that we have an inequality of break divisors
	\[\beta_{(q,e')}(T)\neq\beta_{(q,e)}(T')\]
Thus the rotor-routing and Bernardi torsors disagree on the tree $T$, so that $\rho_q\neq\beta_q$. 
\end{proof}

Combining the results of Theorem \ref{thm:propwit} and Theorem \ref{thm:nopropwit}, we obtain the following Corrollary, verifying the Baker-Wang conjecture.

\begin{Corollary}[Baker-Wang Conjecture]
Given a nonplanar ribbon graph with no multiple edges and no loops, there exists a vertex $q$ for which $\rho_q\neq\beta_q$.
\end{Corollary}

\section{Examples}\label{examples}

We conclude with two examples to address the following questions
\begin{enumerate}
	\item Is the assumption of no multiple edges necessary for the Baker--Wang Conjecture?
	\item For nonplanar ribbon graphs, is the \textit{difference} between the two torsor structures independent on the base vertex?
\end{enumerate}
Our examples show that the answer to question (1) is yes, while the answer to question (2) is no.

Fix a labeling of the $N$ spanning trees of the underlying graph and to consider the actions $\rho_q$ and $\beta_q$ as homomorphisms
	\[
	\begin{aligned}
	\rho_q:\Pic^0(G) &\longrightarrow S_N \\
	\beta_q:\Pic^0(G) &\longrightarrow S_N
	\end{aligned}
	\]
The \textit{difference} between the two torsors \textit{on a given divisor class} $[D]$ can be computed as $\rho_q([D])^{-1}\beta_q([D])\in S_N$. The difference between the two torsor structures is dependent on the base vertex if there exist two vertices $p$ and $q$ and a divisor class $[D]$ such that 
	\[\rho_p([D])^{-1}\beta_p([D])\neq \rho_q([D])^{-1}\beta_q([D])\]

\begin{figure}
\begin{center}
\includegraphics[width=6.25cm, height=2.75cm]{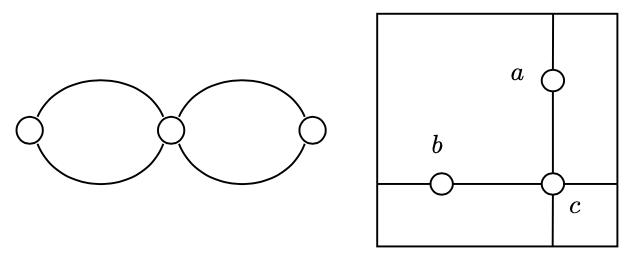} 
\end{center}
\caption{Rounded bowtie graph with the relevant torus embedding. As in prior figures, the ribbon structure is given by the counter-clockwise orientation of the torus.}
\label{fig:roundedbowtie}
\end{figure}

\begin{Example}
Consider the graph depicted in Figure \ref{fig:roundedbowtie}, which we call the \textit{rounded bowtie graph}. For the rounded bowtie graph, we have that $\rho_q=\beta_q$ for all vertices $q$. 

We compute all Bernardi and rotor-routing torsor structures on the graph. This graph has only has four spanning trees, which yields a particularly straightforward collection of computations. 

For a given spanning tree $T$ and a given choice of vertex $q$ and incident edge $e=\{v,q\}$, we first compute the rotor routing action, yielding a permutation $\rho_q([(v)-(q)])\in S_n$. Following this, we calculate the permutation associated to the Bernardi action by computing break divisors $\beta_{(q,e)}(T)$ for all spanning trees $T$. Break divisors will be written as row vectors with coordinates in the order $(a,b,c)$. With these in hand, we translate the break divisors by the divisor $(v)-(q)$, yielding new divisors 
	\begin{equation}
	(v) - (q) + \beta_{(q,e)}(T) \label{translates}
	\end{equation}
For each spanning tree $T$. Finally, we can check for linear equivalence between these divisors and the other break divisors $\beta_{(q,e)}(S)$ for other spanning trees $S$. In particular, if some pair of spanning trees $S$ and $T$ satisfy
	\[(v) - (q) + \beta_{(q,e)}(T) \sim \beta_{(q,e)}(S)\]
then the permutation $\beta_q([(v)-(q)])$ sends the spanning tree $T$ to $S$. Calculating these divisors \eqref{translates} and performing these comparisons for all spanning trees then yields a permutation representation for $\beta_q([(v)-(q)])$. 

In order to verify that $\rho_q=\beta_q$ for all vertices $q$ of the graph in Figure \ref{fig:roundedbowtie}, we perform this process for all choices of vertices $q$ and $v$ of the graph in order to understand the action of divisor classes of $\Div^0(G)$ generators on the trees. There are three non-equivalent choices for $q$ and $v$: one in which the target vertex $q$ has degree four and the vertex $v$ has degree two, one in which $q$ and $v$ both have degree two, and one in which $q$ has degree two and $v$ has degree four. In Figure \ref{fig:bowtiework}, we have included a depiction of the four spanning trees of the rounded bowtie graph, together with their associated Bernardi tours and break divisors for this choice of $q=c$ and $v=a$. 

\begin{figure}
\begin{center}
\includegraphics[width=13.5cm, height=4cm]{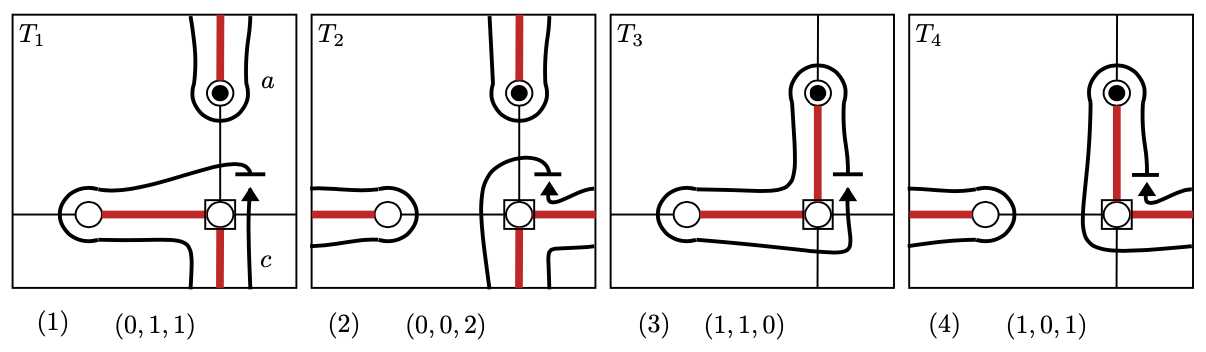} 
\end{center}
\caption{Bernardi and rotor-routing actions on spanning trees of the rounded bowtie graph. The spanning trees $T_i$ are indicated in each of the figures (1) through (4) with thickened edges. The Bernardi tours depicted are associated to the bijection $\beta_{(c,e)}$, where $e=\{a,c\}$. The associated break divisors $\beta_{(c,e)}(T_i)$ are given below each tree.}
\label{fig:bowtiework}
\end{figure}

In calculating the Bernardi action we make use of the Bernardi bijection $\beta_{(c,e)}$, where $e=\{c,a\}$. Consider the divisor $(a)-(c) + \beta_{(c,e)}(T_1)=(1,1,0)$. This is the break divisor $\beta_{(c,e)}(T_3)$, so we conclude that $\beta_c([(a)-(c)])$ maps $T_1$ to $T_3$. Similarly, we see
	
	\[
	\begin{aligned}
	(a) - (c) + \beta_{(c,e)}(T_3) &= (2,1,-1) \sim (0,1,1) = \beta_{(c,e)}(T_1) \\
	(a) - (c) + \beta_{(c,e)}(T_2) &= (1,0,1) = \beta_{(c,e)}(T_4) \\
	(a) - (c) + \beta_{(c,e)}(T_4) &= (2,0,0) \sim (0,0,2) = \beta_{(c,e)}(T_2)
	\end{aligned}
	\]

Together, this implies that we have a permutation representation $\beta_c([(a)-(c)]) = (13)(24)$. Likewise, checking the rotor-routing action $\rho_c([(a)-(c)])$ yields that $\rho_c([(a)-(c)]) = (13)(24) = \beta_c([(a)-(c)])$. By the symmetry of the rounded bowtie graph, it follows that $\rho_c = \beta_c$. Other cases follow similarly and so the associated work is not included.
\end{Example}

We now turn to the issue of the discrepancy between base vertices, illustrating in particular a graph, depicted in Figure \ref{fig:pointedbowtie}, having vertices $p$ and $q$ such that $\rho_p=\beta_p$ while $\rho_q\neq\beta_q$.  

\begin{figure}[b]
\begin{center}
\includegraphics[width=7.25cm, height=3.5cm]{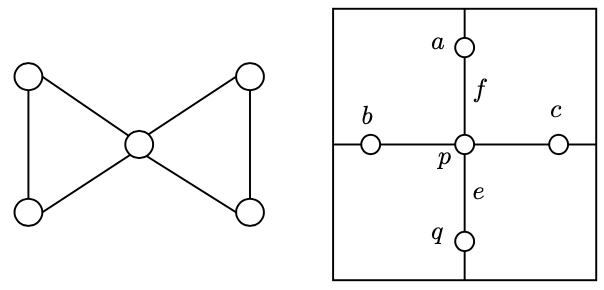} 
\end{center}
\caption{Pointed bowtie graph with the relevant torus embedding. Vertices are labeled $a,b,c,p,q$ and edges $e,f$ are relevant in Example \ref{prop:pq}.}
\label{fig:pointedbowtie}
\end{figure}

\begin{Example}\label{prop:pq}
Consider the graph depicted in Figure \ref{fig:pointedbowtie}, which we call the \textit{pointed bowtie graph}. For the pointed bowtie graph and the two indicated vertices $p$ and $q$ we have $\rho_p=\beta_p$ while $\rho_q\neq\beta_q$. 

We deal with divisor classes of the form $[(z)-(q)]$ for the sake of their simplicity. In Figure \ref{fig:pointedbowtie} all vertices are labeled, and in addition we have included all spanning trees of the graph, together with the labeling to be used in the cycle representation of our actions. We will show that 
	\[\rho_q([(a)-(p)])^{-1}\beta_q([(a)-(p)]) \neq \rho_p([(a)-(p)])^{-1}\beta_p([(a)-(p)]).\]
Before proceeding, note that by Theorem \ref{thm:nopropwit}, we are guaranteed that $\rho_q\neq\beta_q$ for this choice of $q$. We compute now the action of the divisor $[(a)-(p)]$ on each spanning tree using the actions $\rho_q$ and $\beta_q$. Since the base vertex of these actions is $q$ and not $p$, we must first rewrite $[(a)-(p)] = [(a)-(q)] - [(p)-(q)]$, after which one can compute the actions of $[(a)-(q)]$ and $[(p)-(q)]$ under both $\rho_q$ and $\beta_q$. Finally,
	\[
	\begin{aligned}
	\rho_q([(a)-(p)]) &= \rho_q([(a)-(q)]\big[\rho_q([(p)-(q)])\big]^{-1} \\
	\beta_q([(a)-(p)]) &= \beta_q([(a)-(q)]\big[\beta_q([(p)-(q)])\big]^{-1}
	\end{aligned}
	\]

With these permutations in hand, we are able to form the product $\rho_q([(a)-(p)]^{-1}\beta_q([(a)-(p)])$. To start this process, we have produced in Figure \ref{fig:pbowtie_q} the Bernardi tours and break divisors associated to each of the spanning trees. For this example, break divisors will be written as row vectors with coordinates in the order $(a,b,p,c,q)$.With these in hand, we can now compute the cycle representation of the permutation $\beta_q([(p)-(q)])$ using the Bernardi bijection $\beta_{(q,e)}$, where $e$ is the edge $\{p,q\}$. 

Beginning from the spanning tree $T_1$, computation of the translates $[(p)-(q)]+[\beta_{(q,e)}(T_i)]$ and their break divisor representatives yields that the permutation $\beta_q([(p)-(q)])$ has cycle representation $(193)(278)(456)$. On the other hand, it can be verified from performing rotor-routing on the trees $T_1,\ldots,T_9$ that $\rho_q([(p)-(q)])=(123)(459)(678)$. We have
	\[
	\begin{aligned}
	\beta_q([(p)-(q)] &= (193)(278)(456) \\
	\rho_q([(p)-(q)] &= (123)(459)(678) 
	\end{aligned}
	\]

Following this same reasoning, we can obtain cycle representations 
	\[
	\begin{aligned}
	\beta_q([(a)-(q)] &= (139)(287)(465) \\
	\rho_q([(a)-(q)] &= (132)(495)(687) 
	\end{aligned}
	\]
combining the above expressions, it follows that
	\[\rho_q([(a)-(p)])^{-1}\beta_q([(a)-(p)]) = (158)(269)\]
It remains to see that the torsor actions of $\beta_p([(a)-(p)])$ and $\rho_p([(a)-(p)])$ coincide. This is most easily done by taking the same approach as above, utilizing the Bernardi bijection $\beta_{(p,f)}$. These computations are more straightforward than those above, so we leave the work as an exercise. One obtains
	\[
	\begin{aligned}
	\beta_p([(a)-(p)] &= (163)(245)(789) \\
	\rho_p([(a)-(p)] &= (163)(245)(789) 
	\end{aligned}
	\]
In particular, $\rho_p([(a)-(p)])=\beta_p([(a)-(p)])$. Now, note that by the symmetry of the pointed bowtie graph, the coincidence of $\rho_p$ and $\beta_p$ on $[(a)-(p)]$ ensures their coincidence on all generators of $\Div^0(G)$. We now have $\rho_q\neq\beta_q$ while $\rho_p=\beta_p$.

\begin{figure}
\begin{center}
\includegraphics[width=9cm, height=11cm]{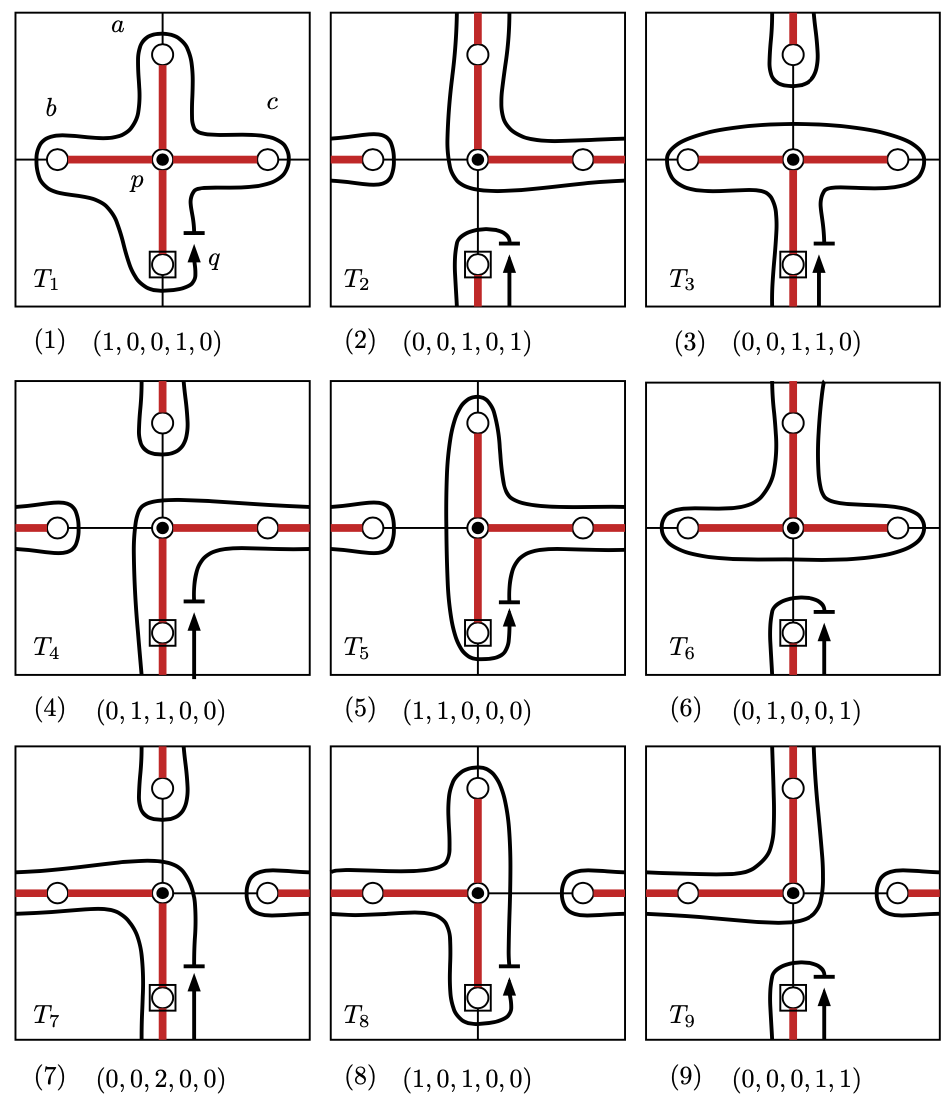} 
\end{center}
\caption{Bernardi and rotor-routing actions on spanning trees of the pointed bowtie graph. The spanning trees $T_i$ are indicated in each of the figures in red. Denoting $e=\{p,q\}$, the break divisors $\beta_{(q,e)}(T_i)$ are given below the corresponding tree with the vertex ordering $(a,b,p,c,q)$.}
\label{fig:pbowtie_q}
\end{figure}

\end{Example}

\bibliography{torsor}
\bibliographystyle{alpha}

\end{document}